\documentclass[11pt]{article}
\usepackage[utf8]{inputenc}
\usepackage{amsmath, amssymb, amsthm}
\usepackage{graphicx}
\usepackage[margin=1in]{geometry}
\usepackage{hyperref}
\usepackage{cite}

\theoremstyle{plain}
\newtheorem{theorem}{Theorem}
\newtheorem{lemma}{Lemma}
\theoremstyle{remark}
\newtheorem*{remark}{Remark}

\newcommand{\Nin}{N^{\mathrm{in}}}
\newcommand{\Nout}{N^{\mathrm{out}}}

\newcommand{\LHS}{\mathrm{LHS}}
\newcommand{\RHS}{\mathrm{RHS}}
\newcommand{\R}{\mathbb{R}}

\newcommand{\half}{\tfrac{1}{2}}

\title{Nonexistence of Whirling-Knight Tours at Half Coil Count\\ for $n\equiv 4,6\pmod 8$}
\author{Shisheng Li\\ \small\texttt{shisheng@mail.ustc.edu.cn}}
\date{\today}

\begin{document}
\maketitle

\begin{abstract}
A \emph{whirling knight's tour} is a Hamiltonian cycle in the digraph
of counter-clockwise knight steps about the centre of an $n\times n$
board; its \emph{coil count} $c$ is the winding number around the
centre.  We prove that no such tour with $c=n/2$ exists when
$n\equiv 4\pmod 8$ ($n\ge 4$) or $n\equiv 6\pmod 8$ ($n\ge 6$),
settling a conjecture of Beluhov.  For each residue class we exhibit
a closed-form Farkas certificate for infeasibility of a cycle-cover
LP relaxation; the two certificates are structurally distinct.


programming).
\end{abstract}

\section{Introduction}\label{sec:intro}

The \emph{whirling knight's digraph} on an $m\times n$ board has the
$mn$ cells as vertices and an arc $(i,j)\to(i',j')$ whenever
$(i-i')^2+(j-j')^2=5$ and the step proceeds counter-clockwise (CCW)
about the pivot $p=\bigl(\tfrac{m-1}{2},\,\tfrac{n-1}{2}\bigr)$.  A
\emph{whirling tour} is a Hamiltonian cycle in this digraph; its
\emph{coil count} $c$ is the winding number of the cycle about $p$.
Equivalently, $c$ is the number of arcs of the cycle that cross the
open vertical ray above $p$ (the \emph{north plumb-line}):  the
winding number around $p$ counts signed crossings of any ray from
$p$, and since every arc is CCW each crossing contributes $+1$,
giving the equivalence.  This object was introduced
by Bennett~\cite{bennett1947} and studied
in~\cite[Exercises~227--231]{knuth-fasc8a}.

Any $m\times n$ whirling tour with $m\le n$ satisfies $m/2\le c\le m$.
The lower endpoint $c=n/2$ on square boards is the focus of this
paper.  An integer-programming search by Beluhov produced $c=n/2$
tours at $n=16$ and $n=18$ but none for $n\equiv 4,6\pmod 8$,
leading him to conjecture nonexistence in those residue classes; the
complementary classes $n\equiv 0,2\pmod 8$ admit $c=n/2$ tours, so
any uniform argument must single out residues $4$ and $6$.  For
related work on leaper Hamilton tours
see~\cite{beluhov-willcocks,beluhov-leaper}.

\medskip\noindent\textbf{Main results.}
\begin{theorem}\label{thm:T1}
For every $n\equiv 6\pmod 8$ with $n\ge 6$, no $n\times n$ whirling
knight's tour with $c=n/2$ coils exists.
\end{theorem}

\begin{theorem}\label{thm:T2}
For every $n\equiv 4\pmod 8$ with $n\ge 4$, no $n\times n$ whirling
knight's tour with $c=n/2$ coils exists.
\end{theorem}

Together, Theorems~\ref{thm:T1} and~\ref{thm:T2} settle Beluhov's
conjecture.

\medskip\noindent\textbf{Method.}
We use the cycle-cover LP relaxation: doubly-stochastic edge
assignments on the CCW knight digraph subject to a single coil-count
equation $\sum_e w_e x_e=c$, where $w_e\in\{0,1\}$ marks crossings of
the north plumb-line.  Infeasibility of this LP rules out a whirling
tour at the given $c$.

For each residue class we exhibit a closed-form Farkas dual
$(\alpha,\beta,\gamma)\in\R^V\times\R^V\times\R$ satisfying
$\alpha_v+\beta_u+\gamma w_e\le 0$ for every arc $e=(u,v)$ and
$\sum_v\alpha_v+\sum_v\beta_v+c\gamma>0$.  Both certificates have
small integer support, but are structurally distinct:
\begin{itemize}
\item For $n\equiv 6\pmod 8$ (Theorem~\ref{thm:T1}), the certificate
puts $+1$-weights on two short blocks of cells flanking the pivot
column.  The combinatorial reading reduces the proof to three
elementary geometric facts about the CCW knight digraph,
contradicting the coil count by a single excess crossing.
\item For $n\equiv 4\pmod 8$ (Theorem~\ref{thm:T2}), the certificate
combines an NE-triangle of cells carrying $\pm 1$ weights according to
$(i+j)$-parity, plus an analogous block-row correction.  The crucial
arithmetic is a parity-counting identity stating that, on the chosen
triangle, the number of cells with odd coordinate sum exceeds the
even count by exactly $2m+1$, where $n=8m+4$.
\end{itemize}

\medskip\noindent\textbf{Layout.}
Section~\ref{sec:setup} sets up the cycle-cover LP and Farkas
certificates.  Section~\ref{sec:lemma} isolates a geometric lemma
used in both proofs.  Sections~\ref{sec:T1} and~\ref{sec:T2} prove
Theorems~\ref{thm:T1} and~\ref{thm:T2}.

\section{Notation and the cycle-cover LP}\label{sec:setup}

\subsection*{Coordinates and the CCW digraph}

Cells of an $n\times n$ board are labelled $(i,j)$ with integer
$0\le i,j\le n-1$;  $i$~is the row index (from top) and $j$~is the
column index (from left).  Each cell is identified with the integer
point at its geometric centre, and the pivot
$(p,q)=\bigl(\tfrac{n-1}{2},\tfrac{n-1}{2}\bigr)$ is the geometric
centre of the board.  Throughout, $n$ is even, so $p$ and $q$ are
half-integers and no cell coincides with the pivot.
Figure~\ref{fig:coords} illustrates the convention on a $4\times 4$
board.

\begin{figure}[ht]
\centering
\includegraphics[width=2.6in]{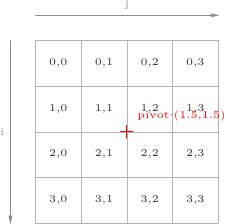}
\caption{Coordinates on a $4\times 4$ board, pivot at $(1.5,1.5)$.}
\label{fig:coords}
\end{figure}

A directed knight arc $(i,j)\to(i',j')$ (with
$(i-i')^2+(j-j')^2=5$) is \emph{counter-clockwise} (CCW) about the
pivot iff
\[
  (i-p)(j'-q) > (i'-p)(j-q).
\]
The \emph{whirling-knight digraph} on $V=\{0,\dots,n-1\}^2$ has all
such arcs; a \emph{whirling tour} is a Hamiltonian directed cycle
in it.

The \emph{coil count} $c$ of a whirling tour~$T$ is the number of
arcs of $T$ that cross the open ray $\{(x,q):x<p\}$, the
\emph{north plumb-line} above the pivot.  For $n$ even, the pivot is
on a half-integer column, so this ray lies between cells and each arc
crosses it $0$ or $1$ times.  Set $w_e=1$ if arc~$e$ crosses the
north plumb-line, else $w_e=0$.  Figure~\ref{fig:ccw} illustrates a
single CCW arc and its crossing point.

\begin{figure}[ht]
\centering
\includegraphics[width=2.6in]{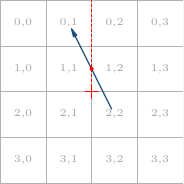}
\caption{The dashed segment is the north plumb-line above the
pivot.  The arrow is the CCW knight arc $(2,2)\to(0,1)$:  the test
gives $(2-1.5)(1-1.5)-(0-1.5)(2-1.5)=0.5>0$.  The arc crosses the
plumb-line at $(1,1.5)$, contributing~$+1$ to the coil count.}
\label{fig:ccw}
\end{figure}

\subsection*{The cycle-cover LP relaxation}

Each whirling tour~$T$ on the $n\times n$ board with $c$~coils
corresponds to an indicator vector $x_e=\mathbf 1_{e\in T}\in\{0,1\}^E$
satisfying the cycle-cover constraints
\[
  \sum_{e\,\text{into}\,v} x_e = 1, \quad
  \sum_{e\,\text{out of}\,v} x_e = 1, \quad \forall v\in V,
\]
together with the coil-count equation $\sum_e w_e\, x_e = c$.
Replacing $\{0,1\}$ by $[0,1]$ yields the linear-programming
relaxation $(\mathrm{LP}_{n,c})$:
\[
\begin{array}{ll}
\text{find } x\in\R^E \\
\text{s.t.} &
  \displaystyle\sum_{e\,\text{into}\,v} x_e = 1,\;
  \displaystyle\sum_{e\,\text{out of}\,v} x_e = 1, \quad \forall v\in V; \\[6pt]
  & \displaystyle\sum_{e\in E} w_e\, x_e = c, \\[6pt]
  & 0\le x_e\le 1, \quad \forall e\in E.
\end{array}
\]
The following lemma is immediate.

\begin{lemma}[Reduction Lemma]\label{lem:red}
If a whirling tour with $c$~coils exists at $(n,c)$, then
$(\mathrm{LP}_{n,c})$ is feasible.  Equivalently,
$(\mathrm{LP}_{n,c})$ infeasible implies that no whirling tour
with $c$~coils on the $n\times n$ board exists.
\end{lemma}

\begin{proof}
Let $T$ be such a tour and set $x_e=\mathbf{1}_{e\in T}\in\{0,1\}\subset[0,1]$;
the box constraint holds trivially.  In a Hamiltonian cycle every
vertex has exactly one incoming and one outgoing arc, so the in/out
degree constraints hold.  The coil-count constraint is the definition
of ``$c$~coils.''
\end{proof}

\subsection*{Farkas certificates}

A \emph{Farkas certificate} (cf.~\cite[\S 7]{schrijver}) for
$(\mathrm{LP}_{n,c})$ is a triple
$(\alpha,\beta,\gamma)\in\R^V\times\R^V\times\R$, where $\alpha$
multiplies the in-degree row of each vertex, $\beta$ the out-degree
row, and $\gamma$ the single coil-count row.  Define the
per-arc and scalar quantities
\[
  \LHS_e := \alpha_v + \beta_u + \gamma\, w_e
  \quad\text{(for $e=(u,v)$)},
  \qquad
  \RHS := \sum_v \alpha_v + \sum_v \beta_v + c\,\gamma.
\]
The certificate is required to satisfy
\begin{align}
  \LHS_e &\le 0, \qquad \forall e\in E, \label{eq:LHS}\\
  \RHS &> 0. \label{eq:RHS}
\end{align}
Multiplying the LP equalities by their multipliers and summing
arc-by-arc yields, for any LP-feasible $x$,
\(
  \sum_e \LHS_e\, x_e = \RHS.
\)
For $x\ge 0$ and a certificate satisfying~\eqref{eq:LHS}
and~\eqref{eq:RHS}, the left side is $\le 0$ while the right side
is $>0$, a contradiction.  Hence the existence of such a certificate
implies infeasibility of $(\mathrm{LP}_{n,c})$.

To illustrate, for $n=3$, $c=2$ the digraph has eight cells (the
centre $(1,1)$ is excluded for odd $n$) and its unique cycle cover
$T_3$ has coil count $3$, so $(\mathrm{LP}_{3,2})$ is infeasible by
direct inspection.  A
two-cell certificate $\alpha_v=\mathbf{1}_{v\in H}$,
$\beta_v=0$, $\gamma=-1$ with $H=\{(0,0),(1,0),(2,0)\}$ confirms
this; see Figures~\ref{fig:n3cycle} and~\ref{fig:n3cert}.

\begin{figure}[ht]
\centering
\begin{minipage}[t]{0.45\linewidth}
\centering
\includegraphics[width=2.4in]{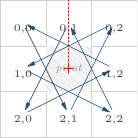}
\caption{The unique whirling tour $T_3$ on the $3\times 3$ board.
The hatched centre is the pivot $(1,1)$, excluded from the digraph.}
\label{fig:n3cycle}
\end{minipage}\hfill
\begin{minipage}[t]{0.45\linewidth}
\centering
\includegraphics[width=2.4in]{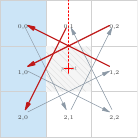}
\caption{The Farkas certificate
$(\alpha=\mathbf 1_H,\beta=0,\gamma=-1)$ for
$(\mathrm{LP}_{3,2})$.  The three red arcs of $T_3$ each cross the
north plumb-line; coil count $3\ne 2=c$, witnessing infeasibility.}
\label{fig:n3cert}
\end{minipage}
\end{figure}

\section{A geometric lemma}\label{sec:lemma}

Both proofs invoke the following geometric fact about CCW knight
arcs touching the pivot column from the north.

\begin{lemma}[Geometric Lemma]\label{lem:G}
Let $n$ be even, $p=\bigl(\tfrac{n-1}{2},\tfrac{n-1}{2}\bigr)$ the
pivot, $h=n/2$, and let $w=(i,j)$ be a cell with $i\le h-2$.
\begin{itemize}
\item[(a)] If $j=h-1$ (column directly west of the pivot), then every
CCW knight arc $u\to w$ satisfies $j_u\in\{h,h+1\}$ (in particular
$j_u\ge h$, strictly east of the pivot column) and crosses the north
plumb-line.
\item[(b)] If $j=h$ (column directly east of the pivot), then every
CCW knight arc $w\to v$ satisfies $j_v\in\{h-1,h-2\}$ and crosses the
north plumb-line.
\end{itemize}
\end{lemma}

\begin{proof}
The CCW test is $(u-p)\times(v-p)>0$, where $\times$ denotes the
$2\times 2$ determinant $(a,b)\times(c,d)=ad-bc$.  Substituting
$u=v-\Delta$ for a knight step
$\Delta\in\{(\pm 1,\pm 2),(\pm 2,\pm 1)\}$ and expanding,
\begin{equation}\label{eq:CCWdelta}
  \Delta_j\,(i_v-p_i) - \Delta_i\,(j_v-p_j) > 0.
\end{equation}
For (a) take $w=v=(i,h-1)$, so $j_v-p_j=-\half$ and
$i_v-p_i=i-(h-\half)\le-\tfrac{3}{2}$ (since $i\le h-2$).
Plugging each of the eight knight steps into~\eqref{eq:CCWdelta}
shows that exactly the four steps
$\Delta\in\{(1,-2),(-1,-2),(2,-1),(-2,-1)\}$ are CCW (the bound
$i_v-p_i\le-\tfrac{3}{2}$ enters monotonically, so the same four
$\Delta$'s remain CCW for every $i\le h-2$); all four have
$\Delta_j\in\{-1,-2\}$, so $j_u=j_v-\Delta_j\in\{h,h+1\}\ge h$.
Since $j_u\ge h>h-\half=p_j>h-1=j_v$, the arc segment crosses the
vertical line $x=p_j$;  the crossing height is the unique
$i$-coordinate at $j=p_j$ along the segment, which by linear
interpolation equals $i_v+\frac{\Delta_i}{\Delta_j}\cdot\half$.  For
the four CCW $\Delta$'s this gives heights $i\pm\tfrac14$ (from
$\Delta_j=-2$) and $i\pm 1$ (from $\Delta_j=-1$);  since
$i\le h-2$, each is at most $h-1<h-\half=p_i$, so the crossing
lies above the pivot and the arc crosses~$N$.  Part~(b) follows by the symmetric calculation
with $j_v-p_j=+\half$ and $i_w-p_i\le-\tfrac{3}{2}$, identifying the
same four CCW steps and giving
$j_v=j_w+\Delta_j\in\{h-1,h-2\}$.
\end{proof}

\begin{remark}\label{rem:sharp}
The hypothesis $i\le h-2$ is sharp for both~(a) and~(b):  for
$w=(h-1,h-1)$ (just NW of the pivot), the arc $(h-3,h-2)\to w$
is CCW (test:
$(-\tfrac52)(-\tfrac12)-(-\tfrac32)(-\tfrac12)=\tfrac12>0$) yet
has $j_u=h-2<h$, violating both the column conclusion of~(a) and
its symmetric counterpart in~(b).  Lemma~\ref{lem:G} therefore
applies to block interiors but not to cells adjacent to the pivot
row.
\end{remark}

\section{Theorem~\ref{thm:T1}: $n\equiv 6\pmod 8$}\label{sec:T1}

Write $n=8m+6$ for $m\ge 0$ and let $c=n/2=4m+3$ and $h=n/2=4m+3$.
Define
\begin{equation}\label{eq:T1support}
\Nin = \{(4k+d,\;h-1)\;:\; 0\le k\le m,\;d\in\{0,1\}\},\quad
\Nout = \{(4k+d,\;h)\;:\; 0\le k\le m,\;d\in\{0,1\}\}.
\end{equation}
Both sets contain $2(m+1)=(n+2)/4$ cells, in two adjacent columns
flanking the pivot column, arranged in $m+1$ ``blocks'' of two rows
separated by gaps of two rows in the northern half of the board.
For $m=0$ ($n=6$) the construction reduces to a single block
$\Nin=\{(0,2),(1,2)\}$, $\Nout=\{(0,3),(1,3)\}$;  the proof below
applies uniformly.  See Figure~\ref{fig:T1blocks}.

\begin{figure}[ht]
\centering
\includegraphics[width=3.4in]{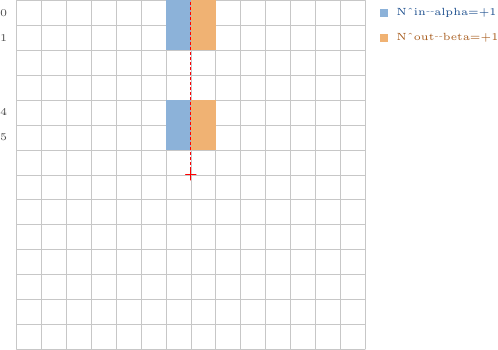}
\caption{The certificate support of Theorem~\ref{thm:T1} for
$n=14$, $m=1$.  Blue cells are $\Nin$ ($\alpha_v=1$);  orange cells
are $\Nout$ ($\beta_v=1$).  For $m=1$ there are two ``blocks'' of two
rows each (rows $0,1$ and $4,5$) separated by a two-row gap.}
\label{fig:T1blocks}
\end{figure}

\medskip\noindent\textbf{The certificate.}
Set
\begin{equation}\label{eq:T1cert}
\alpha_v = \mathbf{1}_{v\in\Nin},\quad
\beta_v = \mathbf{1}_{v\in\Nout},\quad
\gamma = -1.
\end{equation}

\begin{proof}[Proof of Theorem~\ref{thm:T1}]
We verify~\eqref{eq:RHS} and~\eqref{eq:LHS} for the
certificate~\eqref{eq:T1cert}.

For~\eqref{eq:RHS},
\(
  \RHS = |\Nin|+|\Nout|-c = 2(m+1)+2(m+1)-(4m+3) = 1>0.
\)

For~\eqref{eq:LHS}, we check $\LHS_e\le 0$ on every arc by
case analysis on whether the head/tail lies in the support.
\begin{itemize}
\item[(i)] $v\in\Nin$, $u\notin\Nout$:  $\alpha_v=1$, $\beta_u=0$.
$v$ lies in column $h-1$ at row $i_v=4k+d\le 4m+1=h-2$, so
Lemma~\ref{lem:G}(a) applies and the arc crosses~$N$.  Hence
$\LHS_e=1+0-1=0$.
\item[(ii)] $u\in\Nout$, $v\notin\Nin$:  $\beta_u=1$, $\alpha_v=0$.
$u$ lies in column $h$ at row $i_u\le h-2$, so Lemma~\ref{lem:G}(b)
applies and the arc crosses~$N$.  Hence $\LHS_e=0+1-1=0$.
\item[(iii)] $v\in\Nin$ and $u\in\Nout$:  this would give
$\LHS_e=2-1=1$, but knight geometry forbids such an arc.  Indeed it
would have $\Delta j=-1$, forcing $\Delta i=\pm 2$;  but rows in
$\Nin\cup\Nout$ are of the form $4k+d$ with $d\in\{0,1\}$, so the row
gap $i_v-i_u=4(\ell-k)+(d_v-d_u)$ with $|d_v-d_u|\le 1$ is never $\pm 2$.
\item[(iv)] Both endpoints outside the support:  $\alpha_v=\beta_u=0$,
so $\LHS_e=\gamma\cdot[e\text{ crosses }N]\le 0$.
\end{itemize}
Combined with $\RHS=1>0$, the certificate witnesses infeasibility of
$(\mathrm{LP}_{n,n/2})$, and Lemma~\ref{lem:red} concludes the proof.
\end{proof}

\medskip\noindent\textbf{Direct combinatorial restatement.}
The cases (i)--(iii) above amount to three structural facts about
the CCW knight digraph:
\begin{enumerate}
\item[(A)] Every CCW arc into $\Nin$ crosses the north plumb-line.
\item[(B)] Every CCW arc out of $\Nout$ crosses the north plumb-line.
\item[(C)] No CCW arc joins $\Nout$ to $\Nin$.
\end{enumerate}
Suppose for contradiction that a whirling tour~$T$ on the $n\times n$
board has $c=n/2$ coils.  Each cell has exactly one in-arc and one
out-arc in~$T$, so $T$ contains $|\Nin|$ arcs ending in $\Nin$ and
$|\Nout|$ arcs starting in $\Nout$.  By (A)--(B) all these arcs
cross~$N$, and by (C) the two collections are disjoint.  Hence $T$
has at least $|\Nin|+|\Nout|=4(m+1)=4m+4$ N-crossings.  But by
definition $T$ has exactly $c=4m+3$ N-crossings, contradiction.
\hfill$\square$

\section{Theorem~\ref{thm:T2}: $n\equiv 4\pmod 8$}\label{sec:T2}

Write $n=8m+4$ for $m\ge 0$ and let $c=n/2=4m+2$ and $h=n/2=4m+2$.
The certificate combines an \emph{NE-triangle} of cells with
alternating-sign weights and a small \emph{block} of $m+1$ rows.
Define
\begin{align}
T &= \{(i,j):\;0\le i\le h-1,\;h\le j\le n-1,\;i+j\le n-1\},
  \label{eq:T2T}\\
R &= \{4k:\;0\le k\le m\}.
  \label{eq:T2R}
\end{align}
Here $T$ is the upper-right triangle bounded below by the
anti-diagonal $i+j=n-1$ and above/right by the board edges (size
$|T|=h(h+1)/2$); write $T_{\mathrm{even}}$ and $T_{\mathrm{odd}}$
for the cells of $T$ with $i+j$ even and odd, respectively.  $R$
is a set of $|R|=m+1$ ``block rows''.  Note that $R$ consists of
even rows and the block columns $h-1,h$ have one even and one odd
parity, so $R\times\{h\}\subset T_{\mathrm{even}}$ (the relevant
overlap discussed below).
The certificate is
\begin{align}
\alpha_v &= -\mathbf{1}_{v\in T}\cdot\mathbf{1}_{i_v+j_v\,\text{even}}
  \;+\; \mathbf{1}_{(i_v,j_v)=(r,h-1)\text{ for some }r\in R},
  \label{eq:T2cert}\\
\beta_v &= +\mathbf{1}_{v\in T}\cdot\mathbf{1}_{i_v+j_v\,\text{odd}}
  \;+\; \mathbf{1}_{(i_v,j_v)=(r,h)\text{ for some }r\in R},
  \notag\\
\gamma &= -1. \notag
\end{align}
See Figures~\ref{fig:alpha} and~\ref{fig:beta}.

\begin{figure}[ht]
\centering
\begin{minipage}[t]{0.48\linewidth}
\centering
\includegraphics[width=2.6in]{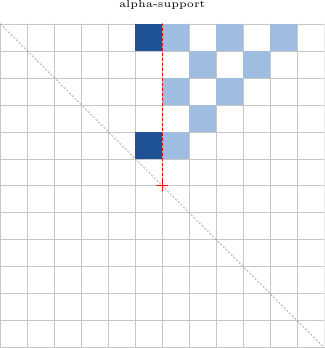}
\caption{$\alpha$-support of certificate~\eqref{eq:T2cert} at
$n=12$, $m=1$.  Light-blue cells ($T_{\mathrm{even}}$) carry
$\alpha_v=-1$;  dark-blue cells in column $h-1=5$ at rows
$R=\{0,4\}$ carry $\alpha_v=+1$.}
\label{fig:alpha}
\end{minipage}\hfill
\begin{minipage}[t]{0.48\linewidth}
\centering
\includegraphics[width=2.6in]{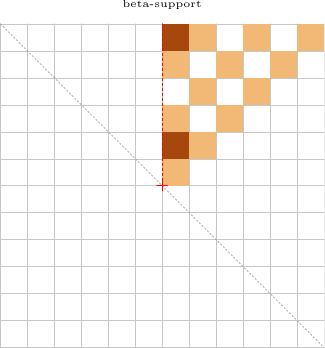}
\caption{$\beta$-support of the same certificate.  Light-orange
cells ($T_{\mathrm{odd}}$) carry $\beta_v=+1$;  dark-orange cells in
column $j=h$ at rows $R$ also carry $\beta_v=+1$.}
\label{fig:beta}
\end{minipage}
\end{figure}

\begin{proof}[Proof of Theorem~\ref{thm:T2}]
Write $T_{\mathrm{even}}, T_{\mathrm{odd}}$ for the $T$-cells with
$i+j$ even or odd.  Row $i$ ($0\le i\le h-1$) contributes $T$-cells
with $j\in[h,\,n-1-i]$, a contiguous strip of length $L_i=h-i$.
Since $h=4m+2$ is even, $i+j$ at $j=h$ has the parity of $i$, and
parities alternate as $j$ increases.  For $i$ even, $L_i$ is even,
so the strip splits evenly:  $L_i/2$ cells of each parity.  For $i$
odd, $L_i$ is odd and the strip starts (at $j=h$) with parity
\emph{odd}, giving $(L_i+1)/2$ odd cells and $(L_i-1)/2$ even cells:
one extra odd cell.  The number of odd $i$ in $[0,h-1]=[0,4m+1]$ is
exactly $2m+1$, so
\begin{equation}\label{eq:parity}
  |T_{\mathrm{odd}}| - |T_{\mathrm{even}}| = 2m+1.
\end{equation}
Hence
\[
  \RHS = \sum_v\alpha_v + \sum_v\beta_v + c\gamma
       = (|T_{\mathrm{odd}}|-|T_{\mathrm{even}}|) + 2|R| - c
       = (2m+1)+2(m+1)-(4m+2) = 1>0,
\]
verifying~\eqref{eq:RHS}.

For~\eqref{eq:LHS}, recall
$\LHS_e = \alpha_v + \beta_u - \mathbf{1}_{e\text{ crosses }N}$.
Check $\LHS_e\le 0$ on every arc by case analysis (the cases below
form a covering, not a partition: arcs with $u\in R\times\{h\}\subset
T_{\mathrm{even}}$ fall under both (ii) and (v), and (v) contains the
correct argument):
\begin{itemize}
\item[(i)] $u,v\in T$:  both endpoints lie in column $j\ge h$, so
the arc stays east of the pivot column and does not cross~$N$.
The block terms cannot contribute:  $R\times\{h-1\}$ has $j<h$
so $v\notin T$ rules out $\alpha_v=+1$;  for $u\in R\times\{h\}$,
Lemma~\ref{lem:G}(b) forces $j_v\in\{h-1,h-2\}<h$, contradicting
$v\in T$.  Knight moves flip $(i+j)\bmod 2$, so the $T$-parts of
$\alpha,\beta$ give $(\alpha_v,\beta_u)\in\{(-1,+1),(0,0)\}$ and
$\LHS_e=0$.

\item[(ii)] $u\in T\setminus(R\times\{h\})$, $v\notin T$,
$v\notin R\times\{h-1\}$ (the case $u\in R\times\{h\}$ is treated
in~(v)):  $\alpha_v=0$.  If $u\in T_{\mathrm{even}}$ then
$\beta_u=0$, so $\LHS_e\le 0$.  Suppose $u\in T_{\mathrm{odd}}$
($\beta_u=+1$);
write $u=(h-1-\alpha,h+\beta)$ with $0\le\beta\le\alpha\le h-1$
(from $u\in T$) and $\alpha+\beta$ even (from $T_{\mathrm{odd}}$).
We claim the arc crosses~$N$.

\emph{Sub-case $i_u\le h-2$ (i.e.\ $\alpha\ge 1$).}
The CCW condition $\Delta_j(2\alpha+1)+\Delta_i(2\beta+1)<0$ is
satisfied by $\Delta\in\{(1,-2),(-1,-2),(-2,-1)\}$ (always for
$u\in T$), $(2,-1)$ (when $\alpha\ge 2\beta+1$), and $(-2,1)$
(when $\alpha\le 2\beta$);  the other three steps are never CCW
for $u\in T$.  For $\beta\ge 2$ each of these CCW $\Delta$'s sends
$u$ to a $v$ with $j_v\ge h$, $i_v\le h-1$, and $i_v+j_v\le n-1$
(when $v$ is on the board), so $v\in T$, contradicting case (ii).
The remaining sub-cases are:
\begin{itemize}
\item $\beta=0$ (i.e.\ $j_u=h$):  Lemma~\ref{lem:G}(b) applies, so
the arc crosses~$N$.
\item $\beta=1$ (i.e.\ $j_u=h+1$):  the only CCW arcs with $j_v<h$
are $\Delta\in\{(1,-2),(-1,-2)\}$, both with $j_v=h-1$;  the
crossing of $j=p_j$ at parameter $t^*=3/4$ has height
$i_u\pm 3/4\le h-5/4<p_i$, so the arc crosses~$N$.
\end{itemize}

\emph{Sub-case $i_u=h-1$ (i.e.\ $\alpha=0$).}  Parity then forces
$\beta=0$, so $u=(h-1,h)$.  Of its four CCW knight arcs, the one
with $v\in T$, namely $(h-1,h)\to(h-3,h+1)$, is excluded by the
case~(ii) hypothesis;  the remaining ones lying on the board are
$(h-1,h)\to(h,h-2)$ and $(h-1,h)\to(h-2,h-2)$ (always), plus
$(h-1,h)\to(h-3,h-1)$ for $h\ge 3$.  Their crossing heights at
$j=p_j$ are $h-\tfrac34$, $h-\tfrac54$, $h-2$ respectively, each
strictly below $p_i=h-\tfrac12$, so all three cross~$N$.

In every sub-case $\LHS_e=0+1-1=0$.

\item[(iii)] $v\in T$, $u\notin T$, $u\notin R\times\{h\}$:
$\beta_u=0$, $\alpha_v\in\{-1,0\}$,
$\LHS_e=\alpha_v-\mathbf{1}_{e\text{ crosses }N}\le 0$.

\item[(iv)] $v=(r,h-1)$, $r\in R$ (block head):  $\alpha_v=+1$.
$v$ sits in column $h-1$ at row $r\le 4m=h-2$, so
Lemma~\ref{lem:G}(a) applies:  the arc crosses~$N$.  We claim
$\beta_u=0$:
(a) $u\notin T_{\mathrm{odd}}$: since $i_v+j_v=r+h-1$ is odd
($r,h$ are both even) and knight moves flip parity, $u$ has even
parity;
(b) $u\notin R\times\{h\}$: such a $u$ would need $\Delta j=-1$,
hence $\Delta i=\pm 2$, but the rows of $R$ are spaced by multiples
of $4$, so $i_v-i_u=4(\ell-k)$ is never $\pm 2$.
Hence $\LHS_e=1+0-1=0$.

\item[(v)] $u=(r,h)$, $r\in R$ (block tail):  $\beta_u=+1$.
$u$ sits in column $h$ at row $r\le h-2$, so Lemma~\ref{lem:G}(b)
applies, giving $j_v\in\{h-1,h-2\}<h$ and the arc crosses~$N$.
Since $j_v<h$, $v\notin T$, ruling out $\alpha_v=-1$ from
$T_{\mathrm{even}}$;  and the $\Delta i\notin\{\pm 2\}$ argument
rules out $v\in R\times\{h-1\}$.  Hence $\alpha_v=0$ and
$\LHS_e=0+1-1=0$.

\item[(vi)] $u,v$ both outside $T$ and the blocks:
$\alpha_v=\beta_u=0$, $\LHS_e=-\mathbf{1}_{e\text{ crosses }N}\le 0$.
\end{itemize}
Combined with $\RHS=1>0$, the certificate witnesses infeasibility of
$(\mathrm{LP}_{n,n/2})$, and Lemma~\ref{lem:red} concludes the proof.
\end{proof}

\medskip\noindent\textbf{Worked example: $n=4$.}
Here $h=2$, $c=2$, $T=\{(0,2),(0,3),(1,2)\}$, $R=\{0\}$.  The cell
$(0,2)\in T_{\mathrm{even}}\cap R\times\{h\}$, so
$\alpha_{(0,2)}=-1$ and $\beta_{(0,2)}=+1$.  The remaining nonzero
values are $\alpha_{(0,1)}=+1$ (from $R\times\{h-1\}$),
$\beta_{(0,3)}=\beta_{(1,2)}=+1$ (from $T_{\mathrm{odd}}$).  Thus
$\sum_v\alpha_v=0$, $\sum_v\beta_v=3$, and $\RHS=0+3-2\cdot 1=1$.
The hypothesis $i_u\le h-2=0$ in Lemma~\ref{lem:G} forces the only
$T$-cells eligible for case~(ii) to lie on row $0$;  the case
analysis specializes correctly, with the boundary sub-case
$i_u=h-1=1$ (here $u=(1,2)$) yielding two on-board CCW arcs
$(1,2)\to(2,0)$ and $(1,2)\to(0,0)$ (the would-be third arc
$(1,2)\to(-1,1)$ is off-board for $h=2$), with crossing heights
$\tfrac54$ and $\tfrac34$ respectively, both below $p_i=\tfrac32$.

\medskip\noindent\textbf{Structural remark.}
Both certificates achieve $\RHS=1$, the smallest positive integer:
$c=n/2$ lies at the LP-feasibility boundary in both classes.  The
two certificates are structurally distinct: T1's $\{0,1\}$-cert
admits the brief combinatorial restatement (A)--(C), whereas T2
requires sign-mixing on an NE-triangle, replacing the König-style
counting of T1 with an alternating-sum identity.

\bigskip
\section*{Acknowledgments}

The author thanks N.~Beluhov for proposing the conjecture and for
valuable feedback, and D.~E.~Knuth for forwarding the problem.

\end{document}